\documentclass[reqno]{amsart}
\usepackage[english]{babel}
\usepackage{amssymb,amsmath,hyperref}
\usepackage{bbold}
\usepackage{amsrefs}
\usepackage{stackrel}
\usepackage[foot]{amsaddr}
\usepackage{mathrsfs}
\usepackage{cleveref, enumerate}

\usepackage{transfer}

\newtheorem{thm}{Theorem}

\newtheorem{cor}[thm]{Corollary}
\newtheorem{defi}[thm]{Definition}
\newtheorem{rem}[thm]{Remark}
\newtheorem{nota}[thm]{Notation}

\newtheorem{ack}[thm]{Acknowledgement}

\newtheorem*{tempo*}{Template}

\newcommand\be{\begin{equation}}
\newcommand\ee{\end{equation}}

\newbox\gnBoxA
\newdimen\gnCornerHgt
\setbox\gnBoxA=\hbox{$\ulcorner$}
\global\gnCornerHgt=\ht\gnBoxA
\newdimen\gnArgHgt

\def\Godelnum #1{%
	\setbox\gnBoxA=\hbox{$#1$}%
	\gnArgHgt=\ht\gnBoxA%
	\ifnum \gnArgHgt<\gnCornerHgt
		\gnArgHgt=0pt%
	\else
		\advance \gnArgHgt by -\gnCornerHgt%
	\fi
	\raise\gnArgHgt\hbox{$\ulcorner$} \box\gnBoxA %
		\raise\gnArgHgt\hbox{$\urcorner$}}
\def\bdefi{\begin{defi}\rm}
\def\edefi{\end{defi}}
\def\bnota{\begin{nota}\rm}
\def\enota{\end{nota}}
\def\brem{\begin{rem}\rm}
\def\erem{\end{rem}}

\def\ATR{\textup{\textsf{ATR}}}

\def\RCA{\textup{\textsf{RCA}}}

\def\RCAo{\textup{\textsf{RCA}}_{0}^{\omega}}
\def\RCAO{\textup{\textsf{RCA}}_{0}^{\Omega}}

\def\WKL{\textup{\textsf{WKL}}}

\def\T{\mathcal{T}}

\def\bye{\end{document}}

\def\P{\textup{\textsf{P}}}

\def\R{{\mathbb  R}}

\def\FAN{\textup{\textsf{FAN}}}
\def\UFAN{\textup{\textsf{UFAN}}}

\def\R{{\mathbb{R}}}
\def\({\textup{(}}
\def\){\textup{)}}

\def\st{\textup{st}}
\def\nst{\textup{nst}}
\def\asa{\leftrightarrow}

\def\di{\rightarrow}

\def\epsa{\varepsilon_{\alpha}}

\def\QFAC{\textup{\textsf{QF-AC}}}

\def\HAC{\textup{\textsf{HAC}}}

\def\INT{\textup{\textsf{int}}}

\numberwithin{equation}{section}
\numberwithin{thm}{section}

\usepackage{comment}

\numberwithin{equation}{section}

\usepackage{comment, enumerate}

\begin{document}
\title{More than bargained for in Reverse Mathematics}

\author{Sam Sanders}
\address{Department of Mathematics (S22), Krijgslaan 281, Ghent University, 9000 Ghent, Belgium \& Munich Center for Mathematical Philosophy, LMU Munich, Germany}
\email{sasander@me.com, http://cage.ugent.be/$\sim$sasander/}
\maketitle
\thispagestyle{empty}
\begin{abstract}
\emph{Reverse Mathematics} (RM for short) is a program in the foundations of mathematics with the aim of finding the minimal axioms required for proving theorems
about countable and separable objects.  
RM usually takes place in \emph{second-order arithmetic} and due to this choice of framework, continuous real-valued functions have to be represented by so-called codes.  
Kohlenbach has shown that the RM-definition of continuity-via-codes constitutes a slight constructive enrichment of the epsilon-delta definition, namely in the 
form of a \emph{modulus} of continuity.  In this paper, we show that the RM-definition of continuity also gives rise to a `nonstandard' enrichment in the form of \emph{nonstandard continuity} from Nonstandard Analysis.  
This observation allows us to (i) establish that RM-theorems related to continuity are implicitly higher-order statements, (ii) prove equivalences between RM-theorems concerning continuity and their associated higher-order versions, and (iii) obtain \emph{explicit} equivalences between higher-order theorems from the equivalence between the corresponding RM-theorems.  
Moreover, we show that it is exactly the \textup{RM}-definition of continuity-via-codes which gives rise to these higher-order phenomena.  
In conclusion, we establish that the practice of coding in RM, designed to obviate higher-type objects, \emph{actually introduces a host of new ones}.  
\end{abstract}



\section{Introduction}\label{intro}
In two words, the topic of this paper is the implicit presence of \emph{higher-order} statements in \emph{second-order} Friedman-Simpson \emph{Reverse Mathematics}.  
In particular, we show that the definition of \emph{continuity-via-codes} used in the latter, gives rise to higher-order statements.    
We first introduce the aforementioned italicised notions.  

\medskip

\emph{Reverse Mathematics} (RM for short) is a program in the foundations of mathematics initiated by Friedman (\cites{fried, fried2}), and developed extensively by Simpson and others (See \cites{simpson1, simpson2} for an overview and introduction).  
The aim of RM is to find the axioms necessary to prove a given theorem of ordinary, i.e.\ about countable and separable objects, mathematics, assuming the `base theory' $\RCA_{0}$, a weak system of computable mathematics.  
RM usually takes place in \emph{second-order arithmetic}, i.e.\ a system of first-order logic with two sorts: natural numbers and sets of the latter (equivalently: Only type 0 and 1 objects are available).  
By contrast, in Kohlenbach's base theory $\RCAo$ for \emph{higher-order} RM (See \cite{kohlenbach2} for details), all finite types are available.  Thus, objects of type `higher than $1$' shall be informally referred to as `higher-order'.   

\medskip

In RM, real numbers are represented by fast-converging Cauchy sequences as in \cite{simpson2}*{II.4.4};  This implies that real-valued functions are not `directly' available in RM (as they have type $1\di 1$).  
To this end, \emph{continuous functions} are represented by (type 1) codes as in \cite{simpson2}*{II.6.1}.  
In \cite{kohlenbach4}*{\S4}, Kohlenbach proves that this RM-definition of continuity involves a slight constructive enrichment of the usual epsilon-delta definition of continuity, namely in the form of a \emph{modulus} of continuity.   
This constructive enrichment should be compared to Simpson's claim to the contrary in \cite{simpson2}*{I.8.9 and IV2.8}.

\medskip

In Section \ref{richment}, we show that the RM-definition of continuity gives rise to a `nonstandard' enrichment, namely that standard RM-continuous functions are \emph{nonstandard} \emph{continuous}, and vice versa, inside a weak system of Nonstandard Analysis based on $\RCAo$.  
In Section \ref{thmbase}, we explore how this observation gives rise to a higher-order statement, namely the existence of a modulus-of-continuity functional, implicit in the base theory $\RCA_{0}$.   
Similar higher-order statements are implicit in other RM-theorems (not necessarily concerning continuity), as explored in Remark~\ref{feralll}.  

\medskip

Now, some readers would perhaps be more easily convinced of the veracity of our claim (that higher-order statements are implicit in second-order RM) if no nonstandard methods were used.  
Hence, in Section \ref{sexpl}, we prove that the statement 
\[
\text{\emph{Every \textup{RM}-continuous function on Cantor space is uniformly \textup{RM}-continuous.}}
\]
is equivalent, in Kohlenbach's higher order RM, to the statement that 
\begin{center}
\emph{There is a type three functional which witnesses the uniform \textup{RM}-continuity of every \emph{RM}-continuous functional on Cantor space.}
\end{center}
This equivalence does not involve (but is inspired by) Nonstandard Analysis.  
We also show that the previous equivalence only goes through because of the use of RM-continuity, as the latter has greatly reduced quantifier complexity compared to the usual definition of continuity.  
Similar equivalences hold for other RM-theorems related to continuity.  We should also point 
out that while the RM-definition of continuity represents continuous type $2$ and $1\di 1$ functions by type $1$ objects, the previous equivalence involves \emph{type $3$ objects} (See Theorems \ref{dorkkk} and \ref{soareyouuuuuu}).        

\medskip

In Section \ref{dirfu}, we push our claim (that higher-order statements are implicit in second-order RM) one step further by deriving \emph{explicit}\footnote{An implication $(\exists \Phi)A(\Phi)\di (\exists \Psi)B(\Psi)$ is \emph{explicit} if there is a term $t$ in the language such that additionally $(\forall \Phi)[A(\Phi)\di B(t(\Phi))]$, i.e.\ $\Psi$ can be explicitly defined in terms of $\Phi$.} equivalences between higher-order principles from equivalences in second-order RM.
Surprisingly, these explicit equivalences are derived using \emph{results in Nonstandard Analysis} from \cite{brie}.  In particular, the results in Section \ref{dirfu} hint at a hitherto unknown \emph{computational} aspect of Nonstandard Analysis, studied further in \cite{samzoo, sambon,samfee}.  

\medskip

In conclusion, the results in this paper suggest that \emph{insisting} on formalising mathematics in second-order arithmetic is self-defeating:  The RM-definition of continuity brings in higher types `through the back door'.  
Note that we do not claim that such a formalisation is pointless:  We merely point out that the reduction in ontological commitment (provided by the use of second-order arithmetic in RM) should not be exaggerated, especially 
since the coding practice of RM gives rise to type $3$ objects, as shown in Theorems \ref{dorkkk} and \ref{soareyouuuuuu}.

%
\section{About and around the base theory $\RCAO$}\label{base}
In this section, we introduce the base theory $\RCAO$ in which we will work.  
In two words, $\RCAO$ is a conservative extension of Kohlenbach's base theory $\RCAo$ from \cite{kohlenbach2} with certain axioms from Nelson's \emph{Internal Set Theory} (\cite{wownelly}) based on the approach from \cites{brie,bennosam}.    
The system $\RCAo$ is in turn a conservative extension of $\RCA_{0}$ for the second-order language by \cite{kohlenbach2}*{Prop.\ 3.1}.  

\subsection{Internal set theory and its fragments}\label{P}
In this section, we discuss Nelson's \emph{internal set theory}, first introduced in \cite{wownelly}, and its fragments from \cite{brie}.  

\medskip

In Nelson's \emph{syntactic} approach to Nonstandard Analysis (\cite{wownelly}), as opposed to Robinson's semantic one (\cite{robinson1}), a new predicate `st($x$)', read as `$x$ is standard' is added to the language of \textsf{ZFC}, the usual foundation of mathematics.  
The notations $(\forall^{\st}x)$ and $(\exists^{\st}y)$ are short for $(\forall x)(\st(x)\di \dots)$ and $(\exists y)(\st(y)\wedge \dots)$.  A formula is called \emph{internal} if it does not involve `st', and \emph{external} otherwise.   
The three external axioms \emph{Idealisation}, \emph{Standard Part}, and \emph{Transfer} govern the new predicate `st';  They are respectively defined\footnote{The superscript `fin' in \textsf{(I)} means that $x$ is finite, i.e.\ its number of elements are bounded by a natural number.} as:  
\begin{enumerate}
\item[\textsf{(I)}] $(\forall^{\st~\textup{fin}}x)(\exists y)(\forall z\in x)\varphi(z,y)\di (\exists y)(\forall^{\st}x)\varphi(x,y)$, for internal $\varphi$ with any (possibly nonstandard) parameters.  
\item[\textsf{(S)}] $(\forall x)(\exists^{\st}y)(\forall^{\st}z)(z\in x\asa z\in y)$.
\item[\textsf{(T)}] $(\forall^{\st}x)\varphi(x, t)\di (\forall x)\varphi(x, t)$, where $\varphi$ is internal,  $t$ captures \emph{all} parameters of $\varphi$, and $t$ is standard.  
\end{enumerate}
The system \textsf{IST} is (the internal system) \textsf{ZFC} extended with the aforementioned external axioms;  
The former is a conservative extension of \textsf{ZFC} for the internal language, as proved in \cite{wownelly}.    

\medskip

In \cite{brie}, the authors study G\"odel's system $\textsf{T}$ extended with special cases of the external axioms of \textsf{IST}.  
In particular, they consider nonstandard extensions of the (internal) systems \textsf{E-HA}$^{\omega}$ and $\textsf{E-PA}^{\omega}$, respectively \emph{Heyting and Peano arithmetic in all finite types and the axiom of extensionality}.       
We refer to \cite{brie}*{\S2.1} for the exact details of these (mainstream in mathematical logic) systems.  
We do mention that in these systems of higher-order arithmetic, each variable $x^{\rho}$ comes equipped with a superscript denoting its type, which is however often implicit.  
As to the coding of multiple variables, the type $\rho^{*}$ is the type of finite sequences of type $\rho$, a notational device used in \cite{brie} and this paper;  Underlined variables $\underline{x}$ consist of multiple variables of (possibly) different type.  

\medskip

In the next section, we introduce the systems $\RCAO$ assuming familiarity with the higher-type framework of G\"odel's $\textsf{T}$ (See e.g.\ \cite{brie}*{\S2.1}).    
\subsection{The base theory $\RCAO$}
In this section, we introduce the system $\RCAO$.  We first discuss some of the external axioms studied in \cite{brie}.  
First of all, Nelson's axiom \emph{Standard part} is weakened to $\HAC_{\INT}$ as follows:
\be\tag{$\HAC_{\INT}$}
(\forall^{\st}x^{\rho})(\exists^{\st}y^{\tau})\varphi(x, y)\di (\exists^{\st}F^{\rho\di \tau^{*}})(\forall^{\st}x^{\rho})(\exists y^{\tau}\in F(x))\varphi(x,y),
\ee
where $\varphi$ is any internal formula.  Note that $F$ only provides a \emph{finite sequence} of witnesses to $(\exists^{\st}y)$, explaining its name \emph{Herbrandized Axiom of Choice}.      
Secondly,  Nelson's axiom idealisation \textsf{I} appears in \cite{brie} as follows:  
\be\tag{\textsf{I}}
(\forall^{\st} x^{\sigma^{*}})(\exists y^{\tau} )(\forall z^{\sigma}\in x)\varphi(z,y)\di (\exists y^{\tau})(\forall^{\st} x^{\sigma})\varphi(x,y), 
\ee
where $\varphi$ is again an internal formula.  
Finally, as in \cite{brie}*{Def.\ 6.1}, we have the following definition.
\bdefi\label{debs}
The set $\T^{*}$ is defined as the collection of all the constants in the language of $\textsf{E-PA}^{\omega*}$.  
The system $ \textsf{E-PA}^{\omega*}_{\st} $ is defined as $ \textsf{E-PA}^{\omega{*}} + \T^{*}_{\st} + \textsf{IA}^{\st}$, where $\T^{*}_{\st}$
consists of the following axiom schemas.
\begin{enumerate}
\item The schema\footnote{The language of $\textsf{E-PA}_{\st}^{\omega*}$ contains a symbol $\st_{\sigma}$ for each finite type $\sigma$, but the subscript is always omitted.  Hence $\T^{*}_{\st}$ is an \emph{axiom schema} and not an axiom.\label{omit}} $\st(x)\wedge x=y\di\st(y)$,
\item The schema providing for each closed term $t\in \T^{*}$ the axiom $\st(t)$.
\item The schema $\st(f)\wedge \st(x)\di \st(f(x))$.
\end{enumerate}
The external induction axiom \textsf{IA}$^{\st}$ is as follows.  
\be\tag{\textsf{IA}$^{\st}$}
\Phi(0)\wedge(\forall^{\st}n^{0})(\Phi(n) \di\Phi(n+1))\di(\forall^{\st}n^{0})\Phi(n).     
\ee
\edefi
For the full system $\P\equiv \textsf{E-PA}^{\omega*}_{\st} +\HAC_{\INT} +\textsf{I}+\textsf{IA}^{\st}$, we have the following theorem.  
The superscript `$S_{\st}$' in the theorem is the syntactic translation defined as follows.
\bdefi
%
%
If $\Phi(\underline{a})$ and $\Psi(\underline{b})$  in the language of $\P$ have interpretations
\be\label{dombu}
\Phi(\underline{a})^{S_{\st}}\equiv (\forall^{\st}\underline{x})(\exists^{\st}\underline{y})\varphi(\underline{x},\underline{y},\underline{a}) \textup{ and } \Psi(\underline{b})^{S_{\st}}\equiv (\forall^{\st}\underline{u})(\exists^{\st}\underline{v})\psi(\underline{u},\underline{v},\underline{b}),
\ee
then they interact as follows with the logical connectives by \cite{brie}*{Def.\ 7.1}:
\begin{enumerate}[(i)]
\item $\psi^{S_{\st}}:=\psi$ for atomic internal $\psi$.  
\item$ \big(\st(z)\big)^{S_{\st}}:=(\exists^{\st}x)(z=x)$.
\item $(\neg \Phi)^{S_{\st}}:=(\forall^{\st} \underline{Y})(\exists^{\st}\underline{x})(\forall \underline{y}\in \underline{Y}[\underline{x}])\neg\varphi(\underline{x},\underline{y},\underline{a})$.  
\item$(\Phi\vee \Psi)^{S_{\st}}:=(\forall^{\st}\underline{x},\underline{u})(\exists^{\st}\underline{y}, \underline{v})[\varphi(\underline{x},\underline{y},\underline{a})\vee \psi(\underline{u},\underline{v},\underline{b})]$
\item $\big( (\forall z)\Phi \big)^{S_{\st}}:=(\forall^{\st}\underline{x})(\exists^{\st}\underline{y})(\forall z)(\exists \underline{y}'\in \underline{y})\varphi(\underline{x},\underline{y}',z)$
\end{enumerate}
\edefi
\begin{thm}\label{consresult}
Let $\Phi(\tup a)$ be a formula in the language of \textup{\textsf{E-PA}}$^{\omega*}_{\st}$ and suppose $\Phi(\tup a)^\Sh\equiv\forallst \tup x \, \existsst \tup y \, \varphi(\tup x, \tup y, \tup a)$. If $\Delta_{\intern}$ is a collection of internal formulas and
\be\label{antecedn}
\P + \Delta_{\intern} \vdash \Phi(\tup a), 
\ee
then one can extract from the proof a sequence of closed terms $t$ in $\mathcal{T}^{*}$ such that
\be\label{consequalty}
\textup{\textsf{E-PA}}^{\omega*} + \Delta_{\intern} \vdash\  \forall \tup x \, \exists \tup y\in \tup t(\tup x)\ \varphi(\tup x,\tup y, \tup a).
\ee
\end{thm}
\begin{proof}
Immediate by \cite{brie}*{Theorem 7.7}.  
\end{proof}
The proofs of the soundness theorems in \cite{brie}*{\S5-7} provide an algorithm $\mathcal{A}$ to obtain the term $t$ from the theorem.  
The following corollary is only mentioned in \cite{brie} for Heyting arithmetic, but is also valid for Peano arithmetic.  
\begin{cor}\label{consresultcor}
If for internal $\psi$ the formula $\Phi(\underline{a})\equiv(\forall^{\st}\underline{x})(\exists^{\st}\underline{y})\psi(\underline{x},\underline{y}, \underline{a})$ satisfies \eqref{antecedn}, then 
$(\forall \underline{x})(\exists \underline{y}\in t(\underline{x}))\psi(\underline{x},\underline{y},\underline{a})$ is proved in the corresponding formula \eqref{consequalty}.  
\end{cor}
\begin{proof}
Clearly, if for $\psi$ and $\Phi$ as given we have $\Phi(\underline{a})^{S_{\st}}\equiv \Phi(\underline{a})$, then the corollary follows immediately from the theorem.  
A tedious but straightforward verification using the clauses (i)-(v) in \cite{brie}*{Def.\ 7.1} establishes that indeed $\Phi(\underline{a})^{S_{\st}}\equiv \Phi(\underline{a})$.  
This verification may be found in \cite{samzoo}*{\S2}.    
\end{proof}
Finally, the previous theorems do not really depend on the presence of full Peano arithmetic.  
Indeed, let \textsf{E-PRA}$^{\omega}$ be the system defined in \cite{kohlenbach2}*{\S2} and let \textsf{E-PRA}$^{\omega*}$ be its definitional extension with types for finite sequences as in \cite{brie}*{\S2}.  
\begin{cor}\label{consresultcor2}
The previous theorem and corollary go through for $\P$ replaced by $\RCAO\equiv \textsf{\textup{E-PRA}}^{\omega*}+\T_{\st}^{*} +\HAC_{\INT} +\textsf{\textup{I}}+\QFAC^{1,0}$.  
\end{cor}
\begin{proof}
The proof of \cite{brie}*{Theorem 7.7} goes through for any fragment of \textsf{E-PA}$^{\omega{*}}$ which includes \textsf{EFA}, sometimes also called $\textsf{I}\Delta_{0}+\textsf{EXP}$.  
In particular, the exponential function is (all what is) required to `easily' manipulate finite sequences.    
\end{proof}
Finally, we note that Ferreira and Gaspar present a system similar to $\P$ in \cite{fega}.  
We plan to study this system in \cite{samfee}, but have no use for it in this paper.

\subsection{Notations}
We finish this section with two remarks on notation.  First of all, we shall use Nelson's notations, as sketched now.      
\begin{rem}[Notations]\label{notawin}\rm
We write $(\forall^{\st}x^{\tau})\Phi(x^{\tau})$ and $(\exists^{\st}x^{\sigma})\Psi(x^{\sigma})$ as short for 
$(\forall x^{\tau})\big[\st(x^{\tau})\di \Phi(x^{\tau})\big]$ and $(\exists^{\st}x^{\sigma})\big[\st(x^{\sigma})\wedge \Psi(x^{\sigma})\big]$.     
We also write $(\forall x^{0}\in \Omega)\Phi(x^{0})$ and $(\exists x^{0}\in \Omega)\Psi(x^{0})$ as short for 
$(\forall x^{0})\big[\neg\st(x^{0})\di \Phi(x^{0})\big]$ and $(\exists x^{0})\big[\neg\st(x^{0})\wedge \Psi(x^{0})\big]$.  Furthermore, if $\neg\st(x^{0})$ (resp.\ $\st(x^{0})$), we also say that $x^{0}$ is `infinite' (resp.\ finite) and write `$x^{0}\in \Omega$'.  
Finally, a formula $A$ is `internal' if it does not involve $\st$, and $A^{\st}$ is defined from $A$ by appending `st' to all quantifiers (except bounded number quantifiers).    
\end{rem}
Secondly, we shall use the usual notations for rational and real numbers and functions as introduced in \cite{kohlenbach2}*{p.\ 288-289} (and \cite{simpson2}*{I.8.1 and II.4.4} for the former).  
\begin{rem}[Real number]\label{keepintireal}\rm
A (standard) real number $x$ is a (standard) fast-converging Cauchy sequence $q_{(\cdot)}^{1}$, i.e.\ $(\forall n^{0}, i^{0})(|q_{n}-q_{n+i})|<_{0} \frac{1}{2^{n}})$.  
We freely use of Kohlenbach's `hat function' from \cite{kohlenbach2}*{p.\ 289} to guarantee that every sequence $f^{1}$ can be viewed as a real.  
Two reals $x, y$ represented by $q_{(\cdot)}$ and $r_{(\cdot)}$ are \emph{equal}, denoted $x=y$, if $(\forall n)(|q_{n}-r_{n}|\leq \frac{1}{2^{n}})$. Inequality $<$ is defined similarly.         
We also write $x\approx y$ if $(\forall^{\st} n)(|q_{n}-r_{n}|\leq \frac{1}{2^{n}})$ and $x\gg y$ if $x>y\wedge x\not\approx y$.  Real-valued functions $F:\R\di \R$ are represented by functionals $\Phi^{1\di 1}$ such that $(\forall x, y)(x=y\di \Phi(x)=\Phi(y))$, i.e.\ equal reals are mapped to equal reals.   
\end{rem}
Thirdly, we use the usual extensional notion of equality.  
\begin{rem}[Equality]\label{equ}\rm
The system $\RCAO$ includes equality between natural numbers `$=_{0}$' as a primitive.  Equality `$=_{\tau}$' for type $\tau$-objects $x,y$ is then defined as:
\be\label{aparth}
[x=_{\tau}y] \equiv (\forall z_{1}^{\tau_{1}}\dots z_{k}^{\tau_{k}})[xz_{1}\dots z_{k}=_{0}yz_{1}\dots z_{k}]
\ee
if the type $\tau$ is composed as $\tau\equiv(\tau_{1}\di \dots\di \tau_{k}\di 0)$.
In the spirit of Nonstandard Analysis, we define `approximate equality $\approx_{\tau}$' as follows:
\be\label{aparth2}
[x\approx_{\tau}y] \equiv (\forall^{\st} z_{1}^{\tau_{1}}\dots z_{k}^{\tau_{k}})[xz_{1}\dots z_{k}=_{0}yz_{1}\dots z_{k}]
\ee
with the type $\tau$ as above.  
All systems under consideration include the \emph{axiom of extensionality} for all $\varphi^{\rho\di \tau}$ as follows:
\be\label{EXT}\tag{\textsf{E}}  
(\forall  x^{\rho},y^{\rho}) \big[x=_{\rho} y \di \varphi(x)=_{\tau}\varphi(y)   \big].
\ee
However, as noted in \cite{brie}*{p.\ 1973}, the so-called axiom of \emph{standard} extensionality \eqref{EXT}$^{\st}$ is problematic and cannot be included in $\RCAO$.   
\end{rem}

\section{Higher-order statements implicit in second-order RM}
In this section, we show that higher-order statements are implicit in second-order RM.  We start by establishing that the RM-definition of continuity 
actually constitutes \emph{nonstandard} continuity (and vice versa) in Section \ref{richment}.  We subsequently show in Section \ref{thmbase} that this `nonstandard' enrichment of continuity gives rise to (equivalent) higher-order statements in the form of a 
modulus-of-continuity functional.  

\subsection{The nonstandard enrichment of continuity}\label{richment}
In this section, we show that the RM-definition of continuity as in \cite{simpson2}*{II.6.1} constitutes a `nonstandard' enrichment of the usual epsilon-delta-definition of continuity.  
In particular, we show that standard functions which are continuous in the sense of RM, i.e.\ given by codes, are also \emph{nonstandard} continuous.    
Conversely, we show that a nonstandard continuous type 2 functional has a code (in the standard world).  

\subsubsection{Definitions and preliminaries}
In this section, we introduce the required definitions and preliminaries. 
First of all, the definition of continuity on Baire space ($\Phi^{2}\in C$ for short) is:
\be\label{contje}
(\forall \alpha^{1})(\exists N^{0})(\forall \beta^{1})(\overline{\alpha}N=_{0}\overline{\beta}N \di \Phi(\alpha)=_{0}\Phi(\beta)).  
\ee
We say that the functional $\Phi^{2}$ is \emph{standard continuous} if it satisfies \eqref{contje}$^{\st}$, and that the functional $\Phi^{2}$ is \emph{nonstandard continuous} if 
\be\label{contje2}
(\forall^{\st} \alpha^{1})(\forall \beta^{1})({\alpha}\approx_{1}{\beta} \di \Phi(\alpha)=_{0}\Phi(\beta)).  
\ee
where $\alpha\approx_{1}\beta$ if $(\forall^{\st}n^{0})[\alpha(n)=\beta(n)]$.  If \eqref{contje} holds limited to binary sequences, we say that $\Phi$ is \emph{continuous on Cantor space}, and write `$\Phi\in C(2^{N})$' for short.   

\medskip

In the next section, we show that standard functions continuous in the sense of RM are also nonstandard continuous as in \eqref{contje2}.
In Theorem~\ref{krof2} below, we also show the `converse', namely that every type~2 functional which is nonstandard continuous as in \eqref{contje2}, has a RM-code (relative to `st').  
By \cite{kohlenbach4}*{Prop.~4.6}, nonstandard continuity thus constitutes a constructive enrichment.  


\medskip

Secondly, with regard to known results, Kohlenbach shows in \cite{kohlenbach4}*{\S4} that the RM-definition of continuity includes a constructive enrichment in the form of a \emph{modulus of \(pointwise\) continuity}, 
in contrast to Simpson's claim (See \cite{simpson2}*{I.8.9 and  IV.2.8}) that Reverse Mathematics analyses theorems `as they stand', i.e.\ without constructive enrichments.  
Notwithstanding this negative result, Kohlenbach also shows in \cite{kohlenbach4}*{\S4} that the enrichment present in \cite{simpson2}*{II.6.1} is in general harmless.  In particular, there is no change to the RM-equivalences of weak K\"onig's lemma.       

\medskip

In more detail, Friedman-Simpson style Reverse Mathematics takes place in (subsystems of) second-order arithmetic, i.e.\ only type $0$ and $1$ (numbers and sets of the latter) objects are available.  
Simpson motivates this choice as follows:
\begin{quote}
[the second-order] language is the weakest one that is rich enough to express and develop the bulk of core mathematics. (\cite{simpson2}*{Preface})
\end{quote}
As a result of this choice of framework, one cannot define real-valued functions `directly' in RM, as the latter objects have type $1\di 1$.  
For this reason, a real-valued continuous function is represented in Reverse Mathematics by a (type 1) \emph{code} as in \cite{simpson2}*{II.6.1}.  
Kohlenbach shows in \cite{kohlenbach4}*{Prop.\ 4.4} that the existence of a code for a continuous functional $\Phi^{2}$, is equivalent to the 
existence of an \emph{associate} for $\Phi$ as in \cite{kohlenbach4}*{Def.\ 4.3}, and equivalent to the existence of a modulus of continuity for $\Phi$.  
Since associates are more amenable to our framework, we shall therefore work with the former, instead of RM-codes.  The definition is as follows.  
\bdefi\label{defke}
The function $\alpha^{1}$ is an associate of a continuous functional $\Phi^{2}$ if:
\begin{enumerate}[(i)]
\item $(\forall \beta^{1})(\exists k^{0})\alpha(\overline{\beta} k)>0$,\label{defitem2}
\item $(\forall \beta^{1}, k^{0})(\alpha(\overline{\beta} k)>0 \di \Phi(\beta)+1=_{0}\alpha(\overline{\beta} k))$.\label{defitem}
\end{enumerate}
\edefi
Note that we assume that every associate is a neighbourhood function as in \cite{kohlenbach4}.  
The range of $\beta$ in the previous definition may be restricted if $\Phi^{2}$ is only continuous on a subspace.  
Finally, if the two items from Definition \ref{defke} only hold relative to `st', then we say that $\alpha^{1}$ is an associate for $\Phi^{2}$ \emph{relative to `\st'}.    

\medskip
%
%

\subsubsection{Nonstandard continuity and known associates}\label{nerdo}
We now show that standard functions continuous in the sense of RM are nonstandard continuous as in \eqref{contje2}, and vice versa.
Our development takes place inside $\RCAO$.  For simplicity, we work over Baire space rather than with real numbers.

\medskip

Since the Reverse Mathematics definition of `continuity-via-codes' implicitly involves a continuous modulus of continuity (again, by \cite{kohlenbach4}*{Prop.\ 4.4}), we shall make the latter explicit.  
Hence, we represent a continuous function $\phi$ on Baire space via a pair of codes $(\alpha^{1}, \beta^{1})$, where $\alpha$ codes $\phi$ and $\beta$ codes its continuous modulus of pointwise continuity $\omega_{\phi}$.      
In more technical detail, $\alpha$ and $\beta$ satisfy 
\[
(\forall \gamma^{1})(\exists N^{0})\alpha(\overline{\gamma}N)>0\wedge (\forall \gamma^{1})(\exists N^{0})\beta(\overline{\gamma}N)>0,
\]
and the values of $\omega_{\phi}$ and $\phi$ at $\gamma^{1}\leq_{1}1$, denoted $\omega_{\phi}(\gamma)$ and $\phi(\gamma)$, are $\beta(\overline{\gamma}k)-1$ and $\alpha(\overline{\gamma}k)-1$ for any $k^{0}$ such that  $\beta(\overline{\gamma}k)>0$ and $\alpha(\overline{\gamma}k)>0$.    
With the previous definitions in place, the following formula makes sense and expresses that $\omega_{\phi}$ is the modulus of continuity of $\phi$:
\be\label{krif12}
(\forall \zeta^{1}, \gamma^{1})(\overline{\zeta}\omega_{\phi}(\zeta)=\overline{\gamma}\omega_{\phi}(\zeta)\di \phi(\zeta)=\phi(\gamma)).
\ee  
In short, the representation of a functional $\phi$ on Baire space via the RM-definition of continuity is equivalent to our representation \eqref{krif12}.  

\medskip

Now, a basic property of any \emph{standard} functional is that it \emph{maps standard inputs to standard outputs}.  
This `standardness' property is a basic axiom\footnote{In particular, the axiom $(\forall^{\st} x^{\sigma}, y^{\sigma\di \tau})(\st(y(x)))$ is  part of $\mathcal{T}_{\st}$ by \cite{bennosam}*{\S2} and \cite{brie}*{\S2}.} of all the systems in \cites{brie, bennosam} and a cornerstone of Nonstandard Analysis.  
Thus, to represent a \emph{standard} continuous function $\phi$ on Cantor space, we should require that $\phi(\gamma)$ and $\omega_{\phi}(\gamma)$ are standard for standard $\gamma^{1}$.  
To accomplish this, we require that $\alpha$ and $\beta$ 
additionally satisfy:
\begin{align}\label{kruks2}
(\forall^{\st} \gamma^{1})(\exists N^{0})(\exists^{\st}K)&[K\geq \alpha(\overline{\gamma}N)>0] \\
&\wedge(\forall^{\st} \gamma^{1})(\exists N^{0})(\exists^{\st}K^{0})[K\geq \beta(\overline{\gamma}N)>0].\notag
\end{align}
Obviously, there are other ways of guaranteeing that $\phi$ and $\omega_{\phi}$ map standard sequences to standard numbers.
Nonetheless, whichever way we guarantee that $\omega_{\phi}$ and $\phi$ are standard for standard input, \eqref{krif12} yields that 
\be\label{krif13}
(\forall^{\st} \zeta^{1})(\exists^{\st}N^{0})(\forall  \gamma^{1})(\overline{\zeta}N=\overline{\gamma}N\di \phi(\zeta)=_{0}\phi(\gamma)),
\ee
since $\omega_{\phi}(\zeta)$ is assumed to be standard for standard $\zeta^{1}$.  Furthermore, we may assume the number $N^{0}$ as in \eqref{krif13} is minimal (though this number depends on the choice of the code for $\phi$).   
Clearly, \eqref{krif13} implies that $\phi$ is also \emph{nonstandard} pointwise continuous, i.e.\ 
\[
(\forall^{\st} \zeta^{1})(\forall  \gamma^{1})({\zeta}\approx_{1}{\gamma}\di \phi(\zeta)=\phi(\gamma)),
\]
which is the `nonstandard enrichment' we mentioned previously.  
Thus, a standard and continuous $\phi$ on Baire space \emph{represented by an associate}, is automatically nonstandard continuous.      
We now prove the `converse' in the following theorem.
\begin{thm}\label{krof2}
In $\RCAO$, a functional $\Phi^{2}$ which is nonstandard continuous on Baire space, has a standard associate relative to `\st'.  
Furthermore, if $\RCAO$ proves that $\Phi^{2}$ is nonstandard continuous on Baire space, a term $t^{1}$ can be extracted from this proof such that $\RCAo$ proves that $t^{1}$ is an associate of $\Phi^{2}$.     
\end{thm}
\begin{proof}
Working in $\RCAO$, nonstandard continuity \eqref{contje2} implies by definition that:
\[
(\forall^{\st} \alpha^{1})(\forall \beta^{1})(\exists^{\st}N^{0})(\overline{\alpha}N=_{0}\overline{\beta}N \di \Phi(\alpha)=_{0}\Phi(\beta)).  
\]
Applying the idealization axiom \textsf{I} for fixed standard $\alpha^{1}$, we obtain 
\be\label{dorkiiii}
(\forall^{\st} \alpha^{1})(\exists^{\st}K^{0})(\forall \beta^{1})(\exists N^{0}\leq_{0}K)(\overline{\alpha}N=_{0}\overline{\beta}N \di \Phi(\alpha)=_{0}\Phi(\beta)).  
\ee
We may remove the bounded quantifier as follows:
\be\label{lolo}
(\forall^{\st} \alpha^{1})(\exists^{\st}K^{0})(\forall \beta^{1})(\overline{\alpha}K=_{0}\overline{\beta}K \di \Phi(\alpha)=_{0}\Phi(\beta)),  
\ee
and apply \textsf{HAC}$_{\textup{\textsf{int}}}$ to \eqref{lolo} obtain a standard functional $\Xi^{1\di 0^{*}}$ such that 
\be\label{kroef}
(\forall^{\st} \alpha^{1})(\exists K^{0}\in \Xi(\alpha))(\forall \beta^{1})(\overline{\alpha}K=_{0}\overline{\beta}K \di \Phi(\alpha)=_{0}\Phi(\beta)).  
\ee
Now define $\Psi(\alpha)$ as the maximum of all $\Xi(\alpha)(i)$ for $i<|\Xi(\alpha)|$.  Then $\Psi^{2}$ is a (standard) modulus of pointwise continuity for $\Phi$, as follows:  
\[
(\forall^{\st} \alpha^{1}, \beta^{1})(\overline{\alpha}\Psi(\alpha)=_{0}\overline{\beta}\Psi(\alpha) \di \Phi(\alpha)=_{0}\Phi(\beta)).  
\]
As in the proof of \cite{kohlenbach4}*{Prop.~4.4}, $\Phi$ now also has an associate $\alpha^{1}$ relative to `st', defined in terms of $\Psi$ and $\Phi$ as follows:
\be\label{frinkkkkkk}
\alpha(\sigma^{0}):=
\begin{cases}
\Phi(\sigma*00\dots)+1 & \Psi(\sigma*00\dots)\leq |\sigma| \\
0 & \text{otherwise}
\end{cases}.
\ee 
Finally, if $\RCAO$ proves \eqref{contje2}, it also proves \eqref{dorkiiii};  Now apply Corollary \ref{consresultcor} to the latter and go through the previous steps to obtain \eqref{frinkkkkkk}.   
\end{proof}
We now speculate why nonstandard and RM-continuity are connected as above.  
\begin{rem}\rm\label{dorkllll}
The correspondence between `continuity-via-an-associate' and nonstandard continuity established above, can be explained as follows:  
Intuitively speaking, both definitions of continuity remove the innermost universal quantifier (involving $\beta^{1}$) in \eqref{contje};  Indeed, 
this reduction in quantifier complexity is literally part of the definition of associate (See item \eqref{defitem2} in Definition \ref{defke}), while nonstandard continuity gives rise to \eqref{lolo}, in which the innermost \emph{internal} universal quantifier (involving $\beta^{1}$) `does not count' from the 
point of view of \HAC$_{\textup{\INT}}$, as the latter applies to \emph{all} internal formulas.  In both cases, the (literal or not) removal of this innermost universal quantifier allows us to obtain a modulus of continuity.       
\end{rem}

\subsection{Nonstandard and higher-order enrichment}\label{thmbase}
In the previous section, we showed that the representation of continuous functions by RM-codes gives rise to nonstandard continuity and vice versa.  
Thus, the following statement is implicit in the RM-definition of continuity in second-order RM:  
\be\label{dorf}
\text{\small\emph{All continuous and standard functions on Baire space are nonstandard cont}.}
\ee
In this section, we show that \eqref{dorf} formulated in the higher-type framework, is equivalent to the existence of a \emph{modulus-of-continuity functional}.  
Since $\RCAo$ cannot prove the existence of a such a functional (See \cite{kohlenbach4}*{Prop.\ 4.4 and 4.6} or \cite{troelstra1}*{\S6, Theorem 2.6.7}),  
\eqref{dorf} gives rise to a \emph{strict higher-order enrichment} of the usual definition of continuity \eqref{contje}.  
In other words, due to the RM-definition of continuity, {higher-order statements are implicit in second-order {RM}}.  

\medskip

To establish the previous claims, consider the following statements:  
\be\label{NC}\tag{\textsf{\textup{NC}}}
(\forall^{\st}\Phi^{2}\in C,\alpha^{1})(\forall \beta^{1})({\alpha}\approx_{1}{\beta} \di \Phi(\alpha)=_{0}\Phi(\beta)).  
\ee
\be\label{MC}\tag{\textsf{\textup{MC}}}
(\exists^{\st} \Psi^{3})(\forall^{\st} \Phi^{2}\in C, \alpha^{1})(\forall \beta^{1})(\overline{\alpha}\Psi(\Phi, \alpha)=_{0}\overline{\beta}\Psi(\Phi, \alpha) \di \Phi(\alpha)=_{0}\Phi(\beta)).  
\ee
Clearly, \ref{NC} is \eqref{dorf} in the higher-type framework and \ref{MC} states the existence of a modulus-of-continuity functional. 

\begin{thm}\label{krof}
In $\RCAO$, we have $ \ref{NC}\asa \ref{MC}$.  
\end{thm}
\begin{proof}
As standard objects are standard for standard input, the reverse implication follows easily.  
For the forward implication, assume the latter principle and obtain, as in the proof of Theorem~\ref{krof2}, that \eqref{lolo} holds for standard and continuous $\Phi^{2}$, i.e.
\be\label{dosey}
(\forall^{\st} \Phi^{2}\in C, \alpha^{1})(\exists^{\st}K^{0})(\forall \beta^{1})(\overline{\alpha}K=_{0}\overline{\beta}K \di \Phi(\alpha)=_{0}\Phi(\beta)).
\ee
Now apply \HAC$_{\textup{\INT}}$ to obtain a standard functional $\Xi^{(2\times 1)\di 0^{*}}$ such that 
\be\label{kroef5}
(\forall^{\st}\Phi^{2}\in C, \alpha^{1})(\exists K^{0}\in \Xi(\Phi, \alpha))(\forall \beta^{1})(\overline{\alpha}K=_{0}\overline{\beta}K \di \Phi(\alpha)=_{0}\Phi(\beta)).  \notag
\ee
Next, define $\Psi(\Phi, \alpha)$ as the maximum of all $\Xi(\Phi, \alpha)(i)$ for $i<|\Xi(\Phi, \alpha)|$.  Then $\Psi^{3}$ is a standard modulus-of-continuity functional as in 
\be\label{kroef6}
(\forall^{\st}\Phi^{2}\in C, \alpha^{1}, \beta^{1})(\overline{\alpha}\Psi(\Phi, \alpha)=_{0}\overline{\beta}\Psi(\Phi, \alpha) \di \Phi(\alpha)=_{0}\Phi(\beta)),
\ee
and the previous formula is exactly \ref{MC}.  
\end{proof}
While the previous theorem provides a higher-order statement implicit in \eqref{dorf}, we would nonetheless like to obtain an equivalence with an internal principle in the previous theorem.  
We now present a way of obtaining such an equivalence.    
\begin{rem}[Internal principles]\rm
First of all, it is shown in \cite{bennosam} that the Transfer principle \emph{limited to formulas without parameters}, denoted \textsf{PF-TP}$_{\forall}$, gives 
rise to a conservative extension of e.g.\ $\RCAO$.  In this way, the functional $(\exists^{2})$ may be assumed to be \emph{standard} if it exists, as its definition is given by a formula without parameters:
\be\tag{$\exists^{2}$}
(\exists \varphi_{0}^{2})(\forall f^{1})\big[ \varphi_{0}(f)=0\asa (\exists x^{0})f(x)=0 \big].
\ee
Thus, $\RCAO+\textsf{PF-TP}_{\forall}$ can be (again conservatively) extended with a new symbol $\varphi_{0}$ and axioms stating that the latter is standard and the (essentially) unique functional as in $(\exists^{2})^{\st}$, if such there is.  
However, this second extension guarantees that $(\exists^{2})^{\st}\di (\exists^{2})$ as $\varphi_{0}$ is no longer a parameter but a symbol from the language.           
The same can be done for any functional which is unique (enough) by definition.  

\medskip

Secondly, if $\phi$ is a function on Cantor space represented by an associate $\alpha^{1}$, with a modulus of continuity $\omega_{\phi}$ as in \eqref{krif12} and \eqref{kruks2}, we may assume that the modulus 
outputs the \emph{least} point of continuity for standard inputs (even in $\RCAO$).  This becomes clear by considering $(\forall^{\st}\gamma^{1})(\exists^{\st}N^{0})\alpha(\overline{\gamma}N)>0$ (a consequence of \eqref{kruks2}) rather than \eqref{krif13}.  
Indeed, the latter allows us to compute the least such $N$, which is -prima facia- not the case for \eqref{krif13} due to the extra $(\forall\beta^{1})$-quantifier.  

\medskip

In other words, the RM-definition of continuity not just constitutes the existence of a modulus of continuity, this modulus also outputs the minimal point of continuity (of course dependent on the choice of the associate representing $\phi$).  
Hence, to reflect the previous observation concerning second-order RM, we may assume a principle $P$ which (relative to `\st') states that a modulus of continuity gives rise to a modulus outputting the \emph{minimal} point of continuity.  

\medskip

Thus, the functional $\Psi$ from \eqref{kroef6} may be assumed to output moduli which yield the minimal point of continuity (assuming $P$).  
Such a functional $\Psi$ is unique and in the same way as discussed at the beginning of this remark, \ref{MC} implies \ref{MC} with all `\st' dropped if \textsf{PF-TP}$_{\forall}$ is given.  
\end{rem}
We also discuss further results similar to Theorem \ref{krof}.
\begin{rem}[Further results]\label{feralll}\rm
One can obtain results similar to Theorem \ref{krof} for the RM of $\WKL_{0}$ (See \cite{simpson2}*{IV}) by considering e.g.\ Heine's theorem.  Due to the RM-definition of continuity, the latter 
implies that all continuous functions on Cantor space are nonstandard \emph{uniform} continuity.  
Similar to Theorem \ref{krof}, the latter nonstandard statement gives rise to a \emph{modulus-of-uniform-continuity} functional, also called \emph{fan functional} (See e.g.\ \cites{kohlenbach2, noortje}).  

\medskip

Another example \emph{not involving continuity} is the Heine-Borel lemma (\cite{simpson2}*{IV.1}), which is the statement that for all sequences of reals $c_{n}, d_{n}$
\be\tag{\textsf{\textup{HB}}}\label{HB}
 (\forall x\in [0,1])(\exists n^{0})(x\in (c_{n}, d_{n}))\di (\exists k^{0})(\forall  x\in [0,1])(\exists n \leq k)(x\in (c_{n}, d_{n})).
\ee
With some effort, one establishes that \ref{HB} implies that a standard open cover of the unit interval has a finite sub-cover which covers \emph{all} of the unit interval, not just the standard numbers.  
The latter nonstandard statement gives rise to a functional witnessing the Heine-Borel lemma. Now, the definition of open set in RM (See \cite{simpson2}*{II.5.6}) guarantees 
that elementhood of an open set is a $\Sigma^{0}_{1}$-formula.  This reduction in quantifier complexity (compared to the usual definition) is the reason we can obtain the aforementioned nonstandard and functional version of the Heine-Borel lemma.     
\end{rem}
In conclusion, we have established that higher-order statements are implicit in second-order RM as a direct consequence of the RM-definition of continuity.  
We obtain more explicit results in the following section, inspired by the implicit results in this section.  
%

\section{Higher-order statements `explicit' in second-order RM}\label{sexpl}
In the previous section, we discussed how nonstandard continuity was implicit in the RM-definition of continuity, and showed that this `nonstandard enrichment' guarantees that 
higher-order statements are implicit in second-order theorems concerning continuity.  In this section, we take a more direct approach and show that the following second-order statement:
\be\label{frlurk}
\text{\emph{Every \textup{RM}-continuous function on Cantor space is uniformly continuous.}}
\ee
is equivalent to the higher-order statement \ref{URC} below, inside $\RCAo$.   This equivalence is only possible because of the use of RM-continuity in \eqref{frlurk}, which greatly reduces the quantifier-complexity (just like nonstandard continuity; See Remark~\ref{dorkllll}).  
Hence, higher-order statements are not \emph{merely implicit} in second-order ones involving continuity, but the latter can be \emph{derived explicitly} from the former.   

\medskip

The following continuity statement is \eqref{frlurk}, again noting that continuity via an RM-code or an associate is equivalent by \cite{kohlenbach4}*{Prop.\ 4.4}.  
\be\label{RC}\tag{\textsf{\textup{RC}}}
(\forall \alpha^{1})\big[(\forall \beta\leq_{1}1)(\exists N^{0})\alpha(\overline{\beta}N)>0 \di (\exists k^{0})(\forall \beta \leq_{1}1)(\exists N\leq k)\alpha(\overline{\beta}N)>0   \big].
\ee
In other words, \ref{RC} is just \cite{simpson2}*{IV.2.2} for Cantor space.  Now consider the following uniform version of \ref{RC}:    
\begin{align}\label{URC}\tag{\textsf{\textup{URC}}}
(\exists \Psi^{3})(\forall \alpha^{1},g^{2})\big[(\forall \beta\leq_{1}1)\alpha(\overline{\beta}&g(\beta))>0 \\
&\di (\forall \beta \leq_{1}1)(\exists N\leq \Psi(g,\alpha))\alpha(\overline{\beta}N)>0   \big].\notag
\end{align}
Note that \ref{URC} is quite natural from the RM point of view, as $g$ plays the role of the modulus-of-continuity which every function represented by an RM code has (See again \cite{kohlenbach4}*{Prop.\ 4.4}).   
We have the following theorem, the meaning of which is discussed below.  
Hunter notes in \cite{hunterphd}*{\S2.1.2} that any choice axiom $\QFAC^{\sigma, 0}$ still results in a conservative extension of $\RCA_{0}^{\omega}$.
Similarly, \textsf{QF-AC} is a weak axiom by \cite{yamayamaharehare}*{Theorem 2.2}.     
\begin{thm}\label{dorkkk}
In $\RCAO$, we have $\ref{URC}^{\st}\asa \ref{RC}^{\st}\asa \WKL^{\st}$.\\
In $\RCAo+\QFAC^{2,1}$, we have $\ref{URC}\asa \ref{RC}\asa \WKL$.
\end{thm}
It is important to note that the following proof only works because in \ref{RC} and \ref{URC}, continuity in the form of an associate (as opposed to \eqref{contje}) is used, greatly reducing overall quantifier-complexity.    
Indeed, this reduction is essential for obtaining \eqref{fromat} (resp.\ \eqref{frugal}), to which $\QFAC^{2,1}$ (resp.\ \HAC$_{\textup{\INT}}$) can be applied.  We now prove Theorem \ref{dorkkk}.    
\begin{proof}
The equivalence $\ref{RC}\asa\WKL$ is straighforward; The same proof goes through for the equivalence relative to `st'.   
Furthermore, since $\QFAC^{1,0}$ is part of $\RCAo$, $\ref{URC}\di \ref{RC}$ is immediate.  The same proof goes through for the implication relative to `st', as \HAC$_{\textup{\INT}}$ implies $\QFAC^{1,0}$ relative to `st'.  
We now prove the remaining implication in the first line of the proof.  

\medskip

Hence, assume \ref{RC}$^{\st}$ and note that we have:
\[
(\forall^{\st} \alpha^{1}, g^{2})\big[(\forall^{\st} \beta\leq_{1}1)\alpha(\overline{\beta}g(\beta))>0 \di (\exists^{\st} k^{0})(\forall^{\st} \beta \leq_{1}1)(\exists N\leq k)\alpha(\overline{\beta}N)>0   \big],
\]
as $g(\beta)$ is standard for standard $\beta^{1}$.  Trivially, we also have 
\[
(\forall^{\st} \alpha^{1}, g^{2})\big[(\forall^{\st} \beta\leq_{1}1)\alpha(\overline{\beta}g(\beta))>0 \di (\exists^{\st} k^{0})(\forall \gamma \leq_{1}1)(\exists N\leq k)\alpha(\overline{\gamma}N)>0   \big],
\]
as $\overline{\gamma}N$ is standard for standard $N^{0}$.  Bringing quantifiers to the front, we obtain
\be\label{frugal}
(\forall^{\st} \alpha^{1}, g^{2})(\exists^{\st}k^{0}, \beta^{1}\leq_{1}1)\big[\alpha(\overline{\beta}g(\beta))>0 \di (\forall \gamma \leq_{1}1)(\exists N\leq k)\alpha(\overline{\gamma}N)>0   \big].
\ee
Since the formula in square brackets in \eqref{frugal} is internal, we may apply \HAC$_{\textup{\INT}}$.  Hence, there is standard $\Xi^{(1\times 2)\di (0^{*}\times 1^{*})}$ such that:
\begin{align}
(\forall^{\st} \alpha^{1}, g^{2})(\exists k^{0}, \beta^{1}\in \Xi(\alpha, g)) \big[\beta\leq_{1}1&\wedge \alpha(\overline{\beta}g(\beta))>0 \label{lds}\\
&\di (\forall \gamma \leq_{1}1)(\exists N\leq k)\alpha(\overline{\gamma}N)>0   \big].\notag
\end{align}
Now define $\Psi(\alpha, g)$ as the maximum of $\Xi(\alpha, g)(1)(i)$ for $i<|\Xi(\alpha, g)(1)|$.  Note that $\Psi$ completely ignores the second component of $\Xi$ (which contains a witness for $\beta^{1}$).    
Hence, since $\Xi(\alpha, g)$ is standard for standard $\alpha^{1}, g^{2}$, the formula \eqref{lds} becomes
\begin{align*}
(\forall^{\st} \alpha^{1}, g^{2})(\exists^{\st} \beta^{1}) \big[\beta\leq_{1}1&\wedge \alpha(\overline{\beta}g(\beta))>0 \\
&\di (\forall \gamma \leq_{1}1)(\exists N\leq \Psi(\alpha, g))\alpha(\overline{\gamma}N)>0   \big].
\end{align*}
The previous formula implies \ref{URC}$^{\st}$, and the first line of the theorem is done.  

\medskip

Finally, we prove the remaining implication in the second line of the theorem.    
We proceed in roughly the same way as in the first paragraph of this proof, but with extra tricks to remove quantifiers prohibiting the use of $\QFAC^{2,1}$ in the internal version of \eqref{frugal}.  
Thus, consider \ref{RC} and obtain the internal version of \eqref{frugal}, i.e.\
\be\label{short}
(\forall \alpha^{1}, g^{2})(\exists k^{0}, \beta^{1}\leq_{1}1)\big[\alpha(\overline{\beta}\tilde{g}(\beta))>0 \di (\forall \gamma \leq_{1}1)(\exists N\leq k)\alpha(\overline{\gamma}N)>0   \big], 
\ee
where $\tilde{g}(\alpha)$ is the least $n\leq g(\alpha)$ such that $\overline{\alpha}n\not\in T$, if such exists and zero otherwise.  
We now (trivially) weaken the consequent of \eqref{short} as follows:
\begin{align}
(\forall \alpha^{1}, g^{2})(\exists k^{0}, \beta^{1}\leq_{1}1)&\big[\alpha(\overline{\beta}\tilde{g}(\beta))>0\label{fromat}\\
& \di (\forall \gamma^{0} \leq_{0^{*}}1)(\exists N\leq k)[|\gamma|=k\di \alpha(\overline{\gamma}N)>0   ]\big].\notag
\end{align}
Hence, applying $\QFAC^{2,1}$ to \eqref{fromat}, we obtain $\Xi^{(1\times 2)\di (0\times 1)}$ witnessing $(k, \beta)$ in \eqref{fromat}.  Again ignoring the second component in $\Xi$ (involving $\sigma$), we obtain \ref{URC}.     
\end{proof}
The restriction to Cantor space in \ref{RC} and \ref{URC} is only for convenience:  In light of the equivalence between weak K\"onig's lemma and 
\emph{bounded K\"onig's lemma} (See \cite{simpson2}*{IV.1.4}), one establishes Corollary \ref{ocrina} in exactly the same way as the theorem.  
Furthermore, the latter corollary contains similar results for \cite{simpson2}*{I.10.3.3} in which a continuous function is bounded on a compact subspace of Baire space.  
\begin{cor}\label{ocrina}
In $\RCAo+\QFAC^{2,1}$, the following are equivalent to $\WKL$\textup{:}
\begin{align}\label{RC2}\tag{\textsf{\textup{RC2}}}
(\forall \alpha^{1}, \gamma^{1})\big[(\forall \beta\leq_{1}\gamma)&(\exists N^{0})\alpha(\overline{\beta}N)>0 \\
& \di (\exists k^{0})(\forall \beta \leq_{1}\gamma)(\exists N\leq k)\alpha(\overline{\beta}N)>0   \big].\notag
\end{align}
\begin{align}\label{URC2}\tag{\textsf{\textup{URC2}}}
(\exists \Psi^{3})(\forall \alpha^{1},\gamma^{1},g^{2})\big[(\forall &\beta\leq_{1}\gamma)\alpha(\overline{\beta}g(\beta))>0 \\
&\di (\forall \beta \leq_{1}\gamma)(\exists N\leq \Psi(g,\alpha, \gamma))\alpha(\overline{\beta}N)>0   \big].\notag
\end{align}
\begin{align}\label{RB}\tag{\textsf{\textup{RB}}}
(\forall \alpha^{1}, \gamma^{1})\big[(\forall \beta\leq_{1}\gamma)&(\exists N^{0})\alpha(\overline{\beta}N)>0 \\
& \di (\exists k^{0})(\forall \beta \leq_{1}\gamma, N^{0})(\alpha(\overline{\beta}N)>0  \di \alpha(\overline{\beta}N)\leq k \big].\notag
\end{align}
\begin{align}\label{URB}\tag{\textsf{\textup{URB}}}
(\exists \Psi^{3})(\forall \alpha^{1},\gamma^{1},g^{2})\big[(\forall &\beta\leq_{1}\gamma)\alpha(\overline{\beta}g(\beta))>0 \\
&\di(\forall \beta \leq_{1}\gamma, N^{0})(\alpha(\overline{\beta}N)>0  \di \alpha(\overline{\beta}N)\leq \Psi(\alpha, g) \big].\notag
\end{align}
\end{cor}
\noindent
In the same vein, we also have the following corollary, where \ref{FMU} is as follows:
\be\label{FMU}\tag{\textsf{\textup{FMU}}}
(\exists \Psi^{3})(\forall \Phi^{2}\in C,\gamma^{1})(\forall \alpha, \beta\leq_{1}\gamma)(\overline{\alpha}\Psi(\Phi, \gamma)=\overline{\beta}\Psi(\Phi, \gamma) \di \Phi(\alpha)=\Phi(\beta)),   
\ee
and \ref{MC}$_{0}$ is \ref{MC} with all `\st' dropped and with a similar extra quantifier $(\forall\gamma^{1})$ guaranteeing a compact domain.   
\begin{cor}\label{slet}
In $\RCAo+\QFAC^{2.1}$, we have $\ref{FMU}\asa \WKL+\ref{MC}_{0}$.  
\end{cor}
\begin{proof}
The forward direction is immediate.  For the reverse direction, again by the proof of \cite{kohlenbach4}*{Prop.\ 4.4}, \ref{MC}$_{0}$ provides a functional $\Xi^{2\di 1}$ such that $\Xi(\Phi, \gamma)$ is an associate for $\Phi^{2}\in C$ on $\{\alpha^{1}:\alpha\leq_{1}\gamma\}$.    
Hence, we obtain $\WKL\di \ref{URC2}\di \ref{FMU}$, assuming \ref{MC}$_{0}$.  
\end{proof}
As noted above, the use of associates in \ref{RC} is essential for obtaining the equivalences in Theorem \ref{dorkkk} and Corollary \ref{ocrina}:  The proof of the former fails if we try to apply it to the following `higher-order' version of \ref{RC}:
\be\label{trut}
(\forall \Phi^{2}\in C(2^{N}))(\exists N^{0})(\forall \alpha^{1},\beta^{1}\leq_{1}1)(\overline{\alpha}N=_{0}\overline{\beta}N \di \Phi(\alpha)=_{0}\Phi(\beta)).  
\ee
Indeed, the antecedent of \eqref{trut} involves \eqref{contje} restricted to Cantor space, which results in a too high quantifier-complexity to apply \textsf{QF-AC}.  Furthermore, we cannot
weaken the consequent of \eqref{trut} as in the proof of the theorem without access to an associate of $\Phi$ (uniformly via a functional).  

\medskip

In conclusion, we emphasise that on one hand, the choice of `continuity via an associate' in \ref{RC}, \ref{RC2}, and \ref{RB}, yields that the latter are automatically 
equivalent to their respective uniform versions \ref{URC}, \ref{URC2}, and \ref{URB}.  
On the other hand, for the `non-associate' version \eqref{trut}, an equivalence with \ref{FMU} is out of the question by Corollary~\ref{slet}, assuming\footnote{Note that by \cite{kohlenbach4}*{Cor.\ 4.11}, $\WKL$ guarantees that each $\Phi^{2}\in C(2^{N})$ \emph{has} an associate on Cantor space, but the corresponding proof is highly non-uniform, i.e.\ a functional providing this associate seems unlikely (without the use of $(\exists^{2})$).  
Furthermore, the proof of \cite{beeson1}*{Lemma, p.\ 65} seems to relativize to oracles, suggesting that $\WKL\not\di \ref{MC}_{0}$.} $\WKL\not\di \ref{MC}_{0}$ over $\RCAo$.  
In other words, the choice of the RM-definition of continuity guarantees that:  
\begin{center}
\emph{Every continuous function on Cantor space is uniformly continuous}, 
\end{center}
is equivalent to the higher-order statement:
\begin{center}
\emph{A functional witnesses the uniform continuity of every continuous function on Cantor space}, 
\end{center}
and such an equivalence does not follow for \eqref{trut}, modulo the non-derivability of \ref{MC}$_{0}$ from $\WKL$.  

\medskip

Finally, with regard to further results, we note that the Heine-Borel lemma \textsf{HB} has the same syntactic structure as \ref{RC}, giving rise to the following theorem.  
Here, \textsf{UHB} is the obvious uniform version of \textsf{HB} as in Remark~\ref{feralll}.   
\begin{cor}
In $\RCAO$, we have $\textup{\textsf{UHB}}^{\st}\asa \ref{HB}^{\st}\asa \WKL^{\st}$.\\
In $\RCAo+\QFAC^{2,1}$, we have $\textup{\textsf{UHB}}\asa \ref{HB}\asa \WKL$.
\end{cor}
\begin{proof}
Similar to the proof of Theorem \ref{dorkkk}.  
\end{proof}
The author shows in \cite{firstHORM} that the uniform version of $\ATR_{0}$ is equivalent to $\ATR_{0}$ itself.  
The results in this section confirm the similarity between $\WKL_{0}$ (in the form of the fan theorem) and $\ATR_{0}$ as pointed out by Simpson in \cite{simpson2}*{I.11.7}.  

\medskip

In conclusion, we have established that RM-theorems like \ref{RC} are equivalent to their higher-order counterpart \ref{URC}, due to the reduced quantifier complexity of the RM-definition of continuity (compared to the usual definition).   
Thus, higher-order statements are not merely implicit in second-order RM, we can establish equivalence between RM-theorems and their higher-order versions.  
In the next section, we push our claims one step further by deriving \emph{explicit\footnote{An implication $(\exists \Phi)A(\Phi)\di (\exists \Psi)B(\Psi)$ is \emph{explicit} if there is a term $t$ in the language such that additionally $(\forall \Phi)[A(\Phi)\di B(t(\Phi))]$, i.e.\ $\Psi$ can be explicitly defined in terms of $\Phi$.\label{dirkske}}} equivalences between higher-order principles from equivalences in second-order RM.

\section{Explicit equivalences implicit in second-order RM}\label{dirfu}
In this section, we push our claim (that higher-order statements are implicit in second-order RM) one step further by deriving \emph{explicit}$^{\ref{dirkske}}$ equivalences between higher-order principles from equivalences in second-order RM.
Furthermore, we show in Remark \ref{foruki} that such results are unique to \emph{second-order} arithmetic.  
\bdefi[Explicit implication]
An implication $(\exists \Phi)A(\Phi)\di (\exists \Psi)B(\Psi)$ is \emph{explicit} if there is a term $t$ in the language such that additionally $(\forall \Phi)[A(\Phi)\di B(t(\Phi))]$, i.e.\ $\Psi$ can be explicitly defined in terms of $\Phi$.  
\edefi
Following \cite{simpson2}*{IV.1.2}, weak K\"onig's lemma is equivalent to the Heine-Borel lemma.  Recall that the fan theorem, denoted $\FAN$, is the classical contraposition of the former.  
Now consider the following explicit versions:
\begin{align}\label{UFAN}\tag{$\UFAN(\Phi)$}
(\forall T^{1}\leq_{1}1,g^{2})\big[(\forall \beta\leq_{1}1)&\overline{\beta}g(\beta)\not \in T \\
&\di (\forall \beta \leq_{1}1)(\exists i\leq \Phi(g))\overline{\beta}i\not\in T   \big].\notag
\end{align}
\begin{align}
(\forall c^{1}_{(\cdot)}, d^{1}_{(\cdot)}, h^{2})\big[
 (\forall x\in [0,1])&(x\in (c_{h(x)}, d_{h(x)})) \tag{$\textsf{\textup{UHB}}(\Psi)$}\\
& \di (\forall  x\in [0,1])(\exists n \leq \Psi(h, c_{(\cdot)}, d_{(\cdot)} ))(x\in (c_{n}, d_{n}))\big], \notag
\end{align}
where we assume that $c_{(\cdot)}, d_{(\cdot)}$ are sequences of rational numbers for simplicity.  
\begin{thm}\label{soareyouuuuuu}
From the proof of $\WKL\asa \ref{HB}$ in $\RCA_{0}$ \(See \cite{simpson2}*{IV.1}\), terms $s, t$ can be extracted witnessing the \emph{explicit} equivalence $\FAN\asa \ref{HB}$ in $\RCAo$, i.e.\  
\be\label{UKI}
(\forall \Phi^{3})[\UFAN(\Phi)\di \textsf{\textup{UHB}}(s(\Phi))] \wedge (\forall \Psi^{3})[\textup{\textsf{UHB}}(\Psi)\di \textsf{\textup{UFAN}}(t(\Psi))].  
\ee
\end{thm}
\begin{proof}
The proof of $\WKL\asa \ref{HB}$ in $\RCA_{0}$ from \cite{simpson2}*{IV.1} trivially goes through relative to `st' in $\RCAO$, i.e.\ the latter proves $\WKL^{\st}\asa \ref{HB}^{\st}$.  
It is now a tedious but straightforward verification that the latter proof also establishes that
\begin{align}
(\forall^{\st}g^{2})(\forall T^{1}\leq_{1}1)\big[(\forall \beta\leq_{1}1)&\overline{\beta}g(\beta)\not \in T \label{FFAN}\\
&\di (\exists^{\st}k )(\forall \beta \leq_{1}1)(\exists i\leq k)\overline{\beta}i\not\in T   \big]\notag
\end{align}
is equivalent over $\RCAO$ to
\begin{align}
(\forall^{\st}h^{2})(\forall c^{1}_{(\cdot)}, d^{1}_{(\cdot)})\big[ (\forall x\in [0,1])&(x\in (c_{h(x)}, d_{h(x)}))\label{FHB} \\
& \di (\exists^{\st}k)(\forall  x\in [0,1])(\exists n \leq k)(x\in (c_{n}, d_{n}))\big].\notag
\end{align}
For completeness, we establish that \eqref{FFAN}$\di$\eqref{FHB} based on the proof of \cite{simpson2}*{IV.1.1}.  
As in the latter, for a binary sequence $s^{0}$ define the rational numbers
\[\textstyle
a_{s}:=\sum_{i<|s|}\frac{s(i)}{2^{i+1}}    \textup{ and }  b_{s}:=a_{s}+\frac{1}{2^{|s|}}, 
\]
and define the tree $T$ by $s\in T\asa (\forall i\leq|s|)\neg(c_{i}<a_{s}<b_{s}<d_{i})$.  Now suppose the standard functional $h$ as in \eqref{FHB} is such that $(\forall x\in [0,1])(x\in (c_{h(x)}, d_{h(x)}))$.     
For $f\leq_{1}1$, define the real $x(f):=\sum_{j=0}^{\infty}\frac{f(j)}{2^{j+1}}$ and note that $a_{\overline{f}n}\leq x(f)\leq b_{\overline{f}n}$ for all $n$.  Next, define the functional $g^{2}$ as follows: $g(f)$ is the least $n\geq h(x(f))$ such that $c_{h(x(f))}<a_{\overline{f}n}<b_{\overline{f}n}<d_{h(x(f))}$.  By definition, we have $(\forall f^{1}\leq_{1}1)\overline{f}g(f)\not\in T$ and \eqref{FFAN} implies $(\forall \beta \leq_{1}1)(\exists i\leq k_{0})\overline{\beta}i\not\in T$ for some standard $k_{0}$.  This number $k_{0}$ also satisfies $ (\forall  x\in [0,1])(\exists n \leq k)(x\in (c_{n}, d_{n}))$.  Thus, we have established \eqref{FFAN}$\di$\eqref{FHB}, and the reverse implication follows in the same way using the proof of \cite{simpson2}*{IV.1.2}.  

\medskip

Next, both \eqref{FFAN} and \eqref{FHB} can trivially be brought into the following form: $(\forall^{\st} l^{2})(\forall S^{1})(\exists^{\st}k^{0})\phi(l, S, k)$, where $\phi$ is internal.  Applying idealisation $\textsf{I}$ to the latter yields the \emph{equivalent} formula $(\forall^{\st} l^{2})(\exists^{\st}k^{0})(\forall S^{1})\phi(l, S, k)$.  
Hence, \eqref{FFAN}$\asa$\eqref{FHB} is equivalent to a formula of the form 
\be\label{tering}
(\forall^{\st}x^{2})(\exists^{\st}y^{0})\varphi(x, y)\asa (\forall^{\st}u^{2})(\exists^{\st} v^{0})\psi(u,v), 
\ee
where $\varphi, \psi$ are again internal.  
Now, $(\forall^{\st}x^{2})(\exists^{\st}y^{0})\varphi(x, y)\di (\forall^{\st}u^{2})(\exists^{\st} v^{0})\psi(u,v)$ trivially implies (since $z(x)$ is standard for standard $x,z$)
\[
(\forall^{\st} z^{3})\big[(\forall^{\st}x^{2})\varphi(x, z(x))\di (\forall^{\st}u^{2})(\exists^{\st} v^{0})\psi(u,v)\big]
\]
which, thanks to an ample serving of classical logic, yields that 
\[
(\forall^{\st} z^{3}, u^{2})(\exists^{\st}x^{2}, v^{0})\big[\varphi(x, z(x))\di \psi(u,v)\big],
\]
where the formula in square brackets is internal.
Applying Corollary \ref{consresultcor} yields a term $t$ such that $\RCAo$ proves 
\[
(\forall z^{3}, u^{2})(\exists x^{2}, v^{0}\in t(z, u))\big[\varphi(x, z(x)))\di \psi(u,v)\big],
\]
Now define $s(z,u)$ as $\max_{i<|t(z, u)(2)|}t(z, u)(2)(i)$, i.e.\ $s$ ignores the components pertaining to $x^{2}$ and takes the maximum of those pertaining to $v^{0}$.  
We have
\[
(\forall z^{3}, u^{2})(\exists x^{2})(\exists v^{0}\leq s(z, u))\big[\varphi(x, z(x)))\di \psi(u,v)\big],
\]
which, again thanks to classical logic, yields
\be\label{tyfus}
(\forall z^{3}\big[(\forall x^{2})\varphi(x, z(x)))\di (\forall v^{2})(\exists v^{0}\leq s(z, u))\psi(u,v)\big].
\ee
Assuming \eqref{tering} is the implication \eqref{FFAN}$\di$\eqref{FHB}, \eqref{tyfus} is exactly the first conjunct of \eqref{UKI}.  The second conjunct of \eqref{UKI} is obtained by repeating the previous steps for 
\[
(\forall^{\st}x^{2})(\exists^{\st}y^{0})\varphi(x, y)\leftarrow (\forall^{\st}u^{2})(\exists^{\st} v^{0})\psi(u,v), 
\]
and the proof is finished.  
\end{proof}
The previous theorem establishes that \emph{explicit}$^{\ref{dirkske}}$ equivalences between higher-order principles may be derived from equivalences in second-order RM.
Although we choose the simplest possible equivalence from the RM of $\WKL_{0}$, the proof of the theorem is still rather messy.  Nonetheless, results similar to \eqref{UKI} may be obtained for other equivalences from the RM of $\WKL_{0}$, using the proof of Theorem \ref{soareyouuuuuu} as a template.   Obvious examples are \ref{RC} and \ref{RB} from the previous theorem.  

\medskip

Note that the previous proof makes essential use of Nonstandard Analysis, in particular the \emph{term extraction algorithm} provided by Theorem \ref{consresult}.    
The explicit equivalence \eqref{UKI} thus hints at a hitherto unknown computation aspect of Nonstandard Analysis.  This will be explored further in \cites{samzoo, sambon, samfee}.  

\medskip

We finish this section with a remark on extensionality.  In particular, we show that the proof of the theorem only reliably goes through for theorems of second-order arithmetic.  
\begin{rem}[Extensionality]\label{foruki}\rm
The proof of $\WKL\asa \ref{HB}$ in $\RCA_{0}$ goes through relative to `st' in $\RCAO$ since all axioms required for the proof in $\RCA_{0}$ are also valid relative to `st' in $\RCAO$.   
However, this does not generalise to proofs in $\RCAo$:  The axiom of extensionality \eqref{EXT} is part of the latter, but $\RCAO$ does not include \eqref{EXT}$^{\st}$, as noted in Remark \ref{equ}.  
Hence, a proof in $\RCAo$ does necessarily goes through in $\RCAO$ relative to `st' if the former invokes \eqref{EXT}.  
However, this implies that results such as \eqref{UKI} can only be `automatically' obtained for second-order statements in general;  We need to track the use of extensionality for higher-order statements proved in $\RCAo$.    
\end{rem}

\begin{ack}\rm
This research was supported by the following funding bodies: FWO Flanders, the John Templeton Foundation, the Alexander von Humboldt Foundation, and the Japan Society for the Promotion of Science.  
The author expresses his gratitude towards these institutions. 
The author would like to thank Ulrich Kohlenbach and Solomon Feferman for repeatedly drawing his attention to the topic of this paper, and for their valuable advice in general.  
\end{ack}

\bye